\documentclass{amsart}
\usepackage{amsmath, amsthm, amssymb, enumerate}
\usepackage{pdfpages}
\usepackage{hyperref}

\usepackage{
	float, 
	wrapfig,
	color,
	subcaption }
\usepackage{microtype}

\newtheorem{theorem}{Theorem}[section]

\newtheorem{lemma}[theorem]{Lemma}
\newtheorem{corollary}[theorem]{Corollary}
\theoremstyle{definition}
\newtheorem{definition}[theorem]{Definition}
\newtheorem{assumption}[theorem]{Assumption}
\newtheorem{remark}[theorem]{Remark}
\newtheorem{example}[theorem]{Example}

\numberwithin{equation}{section}

\newcommand{\Comment}[1]{}
\usepackage{ifthen} 
\usepackage{xifthen} 
\usepackage{twoopt}

\newcommandtwoopt\mucausal[2][][] {
\ifthenelse{\isempty{#1}}
	{\ifthenelse{\isempty{#2}} {\widehat{\boldsymbol{\mu}}^N} {\widehat{\boldsymbol{\mu}}^{N,{#2}}}}
	{\ifthenelse{\isempty{#2}} {\widehat{\boldsymbol{\mu}}^N_{{#1}}} {\widehat{\boldsymbol{\mu}}^{N, {#2}}_{{,#1}}}}
}

\newcommandtwoopt\nucausal[2][][] {
\ifthenelse{\isempty{#1}}
	{\ifthenelse{\isempty{#2}} {\widehat{\boldsymbol{\nu}}^N} {\widehat{\boldsymbol{\nu}}^{N,{#2}}}}
	{\ifthenelse{\isempty{#2}} {\widehat{\boldsymbol{\nu}}^N_{{#1}}} {\widehat{\boldsymbol{\nu}}^{N, {#2}}_{{,#1}}}}
}

\newcommand\causal[2][]{
\ifthenelse{\isempty{#2}}  {{#1}^N_{\square} }
{{#1}^N_{\square {,#2}}} }

\newcommandtwoopt\muempirical[2][][]{
\ifthenelse{\isempty{#1}}
	{\ifthenelse{\isempty{#2}} {\widehat{\mu}^N} {\hat{\mu}^{N,{#2}}}}
	{\ifthenelse{\isempty{#2}} {\widehat{\mu}^N_{{#1}}} {\widehat{\mu}^{N, 					{#2}}_{{#1}}}}
}

\title{Estimating processes in adapted Wasserstein distance}
\author[J.\ Backhoff, D.\ Bartl, M.\ Beiglb\"ock, J.\ Wiesel]{Julio Backhoff, Daniel Bartl, Mathias Beiglb\"ock, Johannes Wiesel}
\thanks{}	
\address{University of Twente, Department of Applied Mathematics}
\email{julio.backhoff@utwente.nl}
\address{Vienna university, Department of Mathematics}
\email{daniel.bartl@univie.ac.at}
\address{Vienna university, Department of Mathematics}
\email{mathias.beiglboeck@univie.ac.at}
\address{Columiba University, Department of Statistics}
\email{johannes.wiesel@columbia.edu }
\date{\today}
\subjclass[2010]{}

\begin{document}
\begin{abstract}
A number of researchers have independently introduced topologies on the set of laws of stochastic processes that extend the usual weak topology. Depending on the respective scientific background this was motivated by applications and connections to various areas (e.g.\ Plug--Pichler - stochastic programming, Hellwig - game theory, Aldous - stability of optimal stopping, Hoover--Keisler - model theory).  Remarkably, all these seemingly independent approaches define the same \emph{adapted weak topology} in finite discrete time. Our first main result is to construct an \emph{adapted} variant of the empirical measure that consistently estimates the laws of stochastic processes in full generality. 

A natural compatible metric for the adapted weak topology is the given by an adapted refinement of the Wasserstein distance, as established in the seminal works of Pflug-Pichler. Specifically,  the adapted Wasserstein distance allows to control the error in stochastic optimization problems, pricing and hedging problems, optimal stopping problems, etc.\ in a Lipschitz fashion. 
The second main result of this article yields quantitative bounds for the convergence of the adapted empirical measure with respect to adapted Wasserstein distance. Surprisingly, we obtain virtually the same optimal rates and concentration results that are known for the classical empirical measure wrt.\ Wasserstein distance.

\medskip

{\color{black}\emph{Mathematics Subject Classification (2010):} 60G42, 90C46, 58E30.}
\end{abstract}

\maketitle

\section{Introduction}
\label{sec:introduction}

For a Polish space $(\mathcal{X}, d)$, the (first order) Wasserstein distance on the set of Borel probabilities $\mathrm{Prob}(\mathcal{X})$, is defined by
\begin{align*}
\mathcal{W}(\mu,\nu):=\inf_{\pi\in\mathrm{Cpl}(\mu,\nu)}\int d(x,y)\,\pi(dx,dy).
\end{align*}
Here $\mathrm{Cpl}(\mu,\nu)$ is the set of couplings between $\mu$ and $\nu$, that is, probabilities $\pi\in\mathrm{Prob}(\mathcal{X}\times\mathcal{X})$ with first marginal $\mu$ and second marginal $\nu$.
The Wasserstein distance is particularly well suited for many stochastic problems involing laws of \emph{random variables}. Accordingly, studying convergence of empirical measures in Wasserstein distance has a long history; we refer to \cite{fournier2015rate} for results and review of the literature.

The situation drastically changes if, instead of random variables, one is interested in laws of \emph{stochastic processes}.
Consider the case of  two timepoints, let $\mathcal{X}=[0,1]\times[0,1]$ and consider the probabilities $\mu=\delta_{(1/2,1)}+\delta_{(1/2,0)}$ and $\nu:=\delta_{(1/2+\varepsilon,1)}+\delta_{(1/2-\varepsilon,0)}$.
Then the discrepancy of $\mu$ and $\nu$ in Wasserstein distance is of order $\varepsilon$, while, considered as laws of stochastic processes, $\mu$ and $\nu$ have very different properties.
 \vspace{-15mm}
 \begin{figure}[H]\label{TraditionalPicture}
    \centering
    \includegraphics[page=1,width=0.45\textwidth]
        {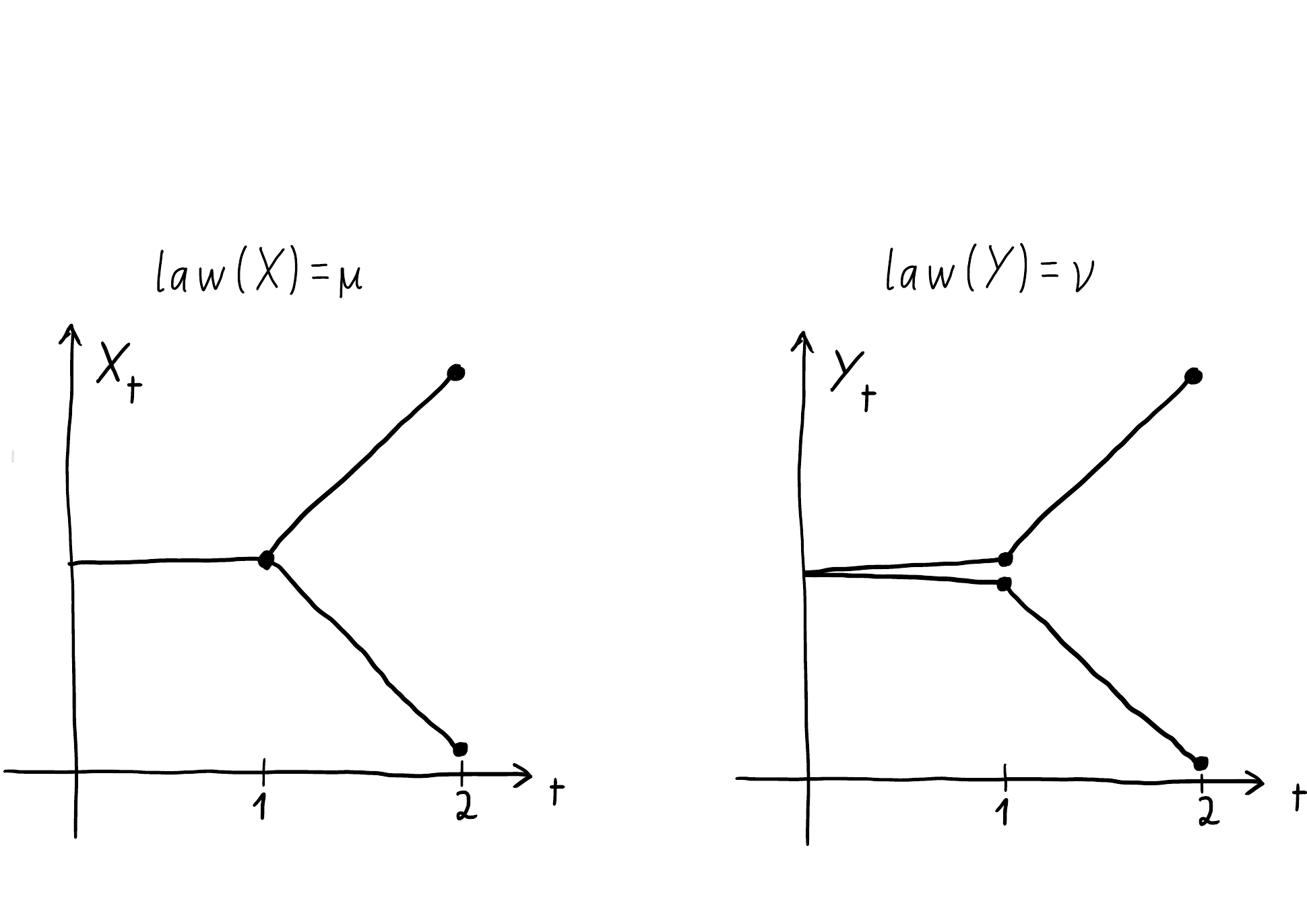} 
  \caption{Close in Wasserstein, very different as stochastic processes.}
 \label{fig:usual.wasserstein}
 \end{figure}
 \vspace{-3mm}
For instance, while no information is available at time $t=1$ in case of $\mu$, the whole future evolution of $\nu$ is known at time $t=1$ already.
In fact, the law of an arbitrary stochastic process can be approximated in classical weak topology  by laws of stochastic processes which are deterministic after the first period.

As already mentioned in the abstract, to overcome this flaw of the Wasserstein distance (or rather, the weak topology), several researchers have introduced adapted versions of the weak topology.
Reassuringly, all these seemingly different definitions  yield the same topology in finite discrete time, see \cite{BaBaBeEd19b}.

 Below we present an  \emph{adapted} extension of the classical Wasserstein distance which induces this topology.
In analogy to its classical counterpart, it turns out to be  particularly well suited to obtain a quantitative control of stochastic optimization problems, see e.g.\ \cite{glanzer2019incorporating, BaBaBeEd19a}.

\subsection{Causality and adapted / nested Wasserstein distance}
Fix $d \in \mathbb{N}$, which we interpret as the dimension of the state space, denote $T\geq 2$ the number of time points under consideration, and let $\mu, \nu$ be  Borel probability measures on $\mathcal{X}=([0,1]^d)^T$. 
In order to account for the temporal structure of stochastic processes, it is necessary to restrict to couplings of probability measures that are non-anticipative in a specific sense: 

Write $X=(X_1, \ldots, X_T)$, $Y=(Y_1, \ldots, Y_T)$ for the projections $X, Y\colon \mathcal{X} \times \mathcal X \to \mathcal X$ onto the first respectively the second coordinate. 
A coupling $\pi\in \mathrm{Cpl}(\mu, \nu)$ is called \emph{causal} (in the language of Lassalle \cite{La18a}) if for all $t< T$ the following holds: 
\begin{align}\label{CausalityDef}
\text{given } (X_1, \ldots, X_t), \quad Y_t \text{ and } (X_{t+1}, \ldots, X_T) \text{ are $\pi$-independent.}
\end{align}
 That is to say, in order to predict $Y_t$, the only information relevant in $X_1,\dots,X_T$ is already contained in $X_1,\dots,X_t$. 
 
The concept of causal couplings  is a suitable extension of   adapted processes: a process $Z=(Z_1, \ldots, Z_T)$ on $\mathcal X$ is adapted with respect to the natural filtration if each $Z_t$ depends only on the values of $X_1, \ldots, X_t$ (so in particular is conditionally independent of $X_{t+1}, \ldots, X_T$ given $X_1,\dots,X_t$). Property \eqref{CausalityDef} represents a counterpart of adaptedness on the level of couplings rather than processes. 

Similarly we call a coupling \emph{anti-causal} if it satisfies \eqref{CausalityDef} with the roles of $X$ and $Y$ interchanged and finally we call $\pi$ \emph{bi-causal} if it is causal as well as anti-causal. We denote the set of bi-causal couplings with marginals $\mu, \nu$ by $\mathrm{Cpl}_{\mathrm{bc}}(\mu, \nu)$.

\begin{definition}[Adapted Wasserstein distance / nested distance]
	The adapted Wasserstein distance (or nested distance) $\mathcal{AW}$ on $\mathrm{Prob}(([0,1]^d)^T)$ is defined as 
	\begin{align}\label{AWdef}
	\mathcal{AW}(\mu,\nu) := \inf_{\pi\in\mathrm{Cpl}_{\mathrm{bc}}(\mu,\nu)}\int \sum_{t=1}^T|x_t-y_t|\,\pi(dx,dy).
	\end{align}
\end{definition}

Bi-causal couplings and the corresponding transport problem were considered by R\"uschendorf \cite{Ru85} under the name `Markov-constructions'. Independently, the concept was introduced by Pflug-Pichler \cite{PfPi12} who realized the full potential of the modified Wasserstein distance in the context of stochastic multistage optimization problems, see also \cite{PfPi14, PfPi15, PfPi16, glanzer2019incorporating}. Pflug-Picher refer to \eqref{AWdef} as \emph{process distance} or \emph{nested distance}. The latter name is motivated by an alternative representation of \eqref{AWdef} through a  dynamic programming principle. For notational simplicity we state it here only for the case $T=2$ where one obtains the representation
\begin{align}
\label{AWdef2} 
\mathcal{AW}(\mu,\nu)
=\inf_{\gamma\in\mathrm{Cpl}(\mu_1,\nu_1) } \int |x_1-y_1| + \mathcal{W}(\mu_{x_1},\nu_{y_1}) \,\gamma(dx_1,dy_1).
\end{align}
Here (and in the rest of this article), for $\mu\in\mathrm{Prob}(([0,1]^d)^T)$ and $1\leq t\leq T-1$, we denote by $\mu_1$ the first marginal of $\mu$ and by $\mu_{x_1,\dots,x_t}$ the disintegration of $\mu$, that is,
\[ \mu_1(\cdot)=P[X_1\in \, \cdot \, ]
\quad\text{and}\quad 
\mu_{x_1,\dots,x_t}(\cdot):=P[ X_{t+1} \in \, \cdot \,| X_1=x_1,\dots,X_t=x_t]\]
for all $(x_1,\dots,x_t)\in([0,1]^d)^t$, where $X$ is a process with law $\mu$.
Informally, the representation in \eqref{AWdef2} asserts that two probabilities are close in adapted Wasserstein distance if  (and only if) besides their marginals, also their kernels  are similar. This is exactly what fails in the example presented in Figure \ref{fig:usual.wasserstein}.

\subsection{Main results}

Let $\mu$ be a Borel probability measure on $([0,1]^d)^T$ 
capturing the true dynamics of the process under consideration. 
Furthermore let $(X^n)_{n\in\mathbb{N}}$ be an i.i.d.\ sample of $\mu$, defined on some fixed (sufficiently rich) abstract  probability space $(\Omega,\mathcal{F},P)$, i.e., each $X^n=(X_1^n,\dots,X_T^n)$ is distributed according to $\mu$.

\begin{definition}[Adapted empirical measure]
\label{def:adapted.empirical.measure}
	Set $r=(T+1)^{-1}$ for $d=1$ and $r=(dT)^{-1}$ for $d\geq 2$.
	For all $N\geq 1$, \emph{partition} the cube $[0,1]^d$ into the \emph{disjoint union} of $N^{rd}$ cubes with edges of length $N^{-r}$ and let $\varphi^N\colon[0,1]^d\to[0,1]^d$ map each such small cube to its center.
	Then define
	\[ \mucausal:=\frac{1}{N}\sum_{n=1}^N \delta_{\varphi^N(X_1^n),\dots,\varphi^N(X_T^n)}. \]
	for each $N\geq 1$. 
	We call $\mucausal$ the {\emph {adapted empirical measure}}.
\end{definition}
 \vspace{-5mm}
 \begin{figure}[H]\label{TraditionalPicture}
    \centering
    \includegraphics[page=1,width=0.75\textwidth]
        {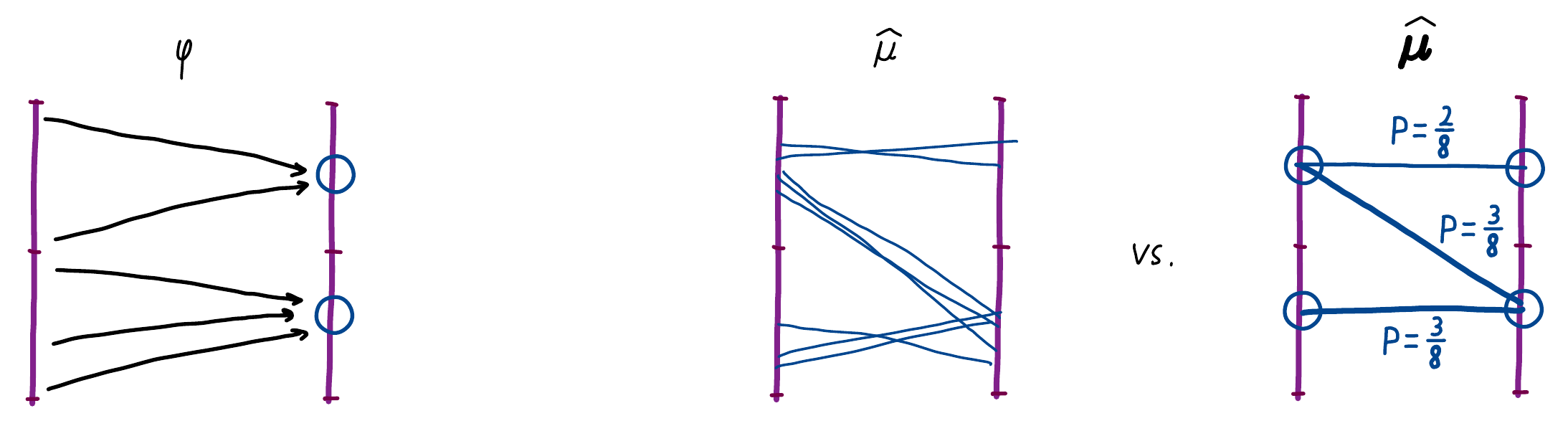} 
  \caption{The map $\varphi$ and comparison of empirical vs.\ adapted empirical measure for $d=1, T=2, N=8, r=1/3$.}
 \label{fig:empirical}
 \end{figure}
 \vspace{-3mm}

That is, the function $\varphi^N$ satisfies $\sup_{u\in[0,1]^d} |u-\varphi^N(u)|\leq CN^{-r}$ and its range $\varphi^N([0,1]^d)$ consist of $N^{rd}$ points.
If $\varphi^N$ were the identity, then $\mucausal$ would be the \emph{(classical) empirical measure}, which we denote by $\muempirical$. 

It was first noted by Pflug-Pichler in \cite{PfPi16} that, in contrast the classical Wasserstein distance,  $\mathcal{AW}(\mu,\muempirical)$ does \emph{not} tend to 0 for generic choices of $\mu$ (cf.\ Remark \ref{rem:classical.empirical.doesnt.converge} below). For similar reasons, an optimal stopping problem solved under $\muempirical$ as a reference model, need not converge a.s.\ to the same problem under model $\mu$ (cf.\ Example \ref{ex:opt.stop.not.cont.usual.empirical}).

Our first main Theorem is the following consistency result for the adapted empirical measure:

\begin{theorem}[Almost sure convergence]
\Comment{For this result we could also replace $N^{rd}$ by any sublinear function, should we make a remark somewhere? MB: It think we should indeed make such a remark, maybe at the end of the introduction? We should also stress that our empirical measure is ``invariant'' under permutation of the time-indices and hence allows for consistent estimation in the `back to the future' topology.}
\label{thm:almost.sure.convergence}
	The adapted empirical measures is a strongly consistent estimator, that is, 
	\[  \lim_{N\to\infty} \mathcal{AW}(\mu,\mucausal) = 0 \]	
	$P$-almost surely.
\end{theorem}

In particular, as $\mathcal{W}\leq \mathcal{AW}$ by definition, it follows that the adapted empirical measure converges in the usual weak topology as well. 

In order to quantify the speed of convergence, we assume the following regularity property for the remainder of this section.

\begin{assumption}[Lipschitz kernels]
\label{ass:lipschitz.kernel}
	There is a version of the ($\mu$-a.s.\ uniquely defined) disintegration such that for every $1\leq t\leq T-1$ the mapping
	\begin{align*}
	([0,1]^d)^t 
	\ni (x_1, \dots, x_t) 
	\mapsto \mu_{x_1,\dots,x_t}
	\in \mathrm{Prob}([0,1]^d)
	\end{align*}
	is Lipschitz continuous, where $\mathrm{Prob}([0,1]^d)$ is endowed with its usual Wasserstein distance $\mathcal{W}$.
\end{assumption} 

\begin{theorem}[Average rate]
\label{thm:rates.unit.cube}
	Under Assumption \ref{ass:lipschitz.kernel}, there is a constant $C>0$ such that
	\begin{align}
	\label{eq:mean.speed.rate}
	\begin{split}
	E\Big[ \mathcal{AW}(\mu,\mucausal)\Big]
	&\leq C \cdot
	\begin{cases}
	N^{-1/(T+1)} &\text{for } d=1,\\
	N^{-1/(2T)}\log(N+1) &\text{for } d=2,\\
	N^{-1/(dT)} &\text{for } d\geq 3,
	\end{cases} \\
	&=:C \cdot	\mathop{\mathrm{rate}}(N) 
	\end{split}
	\end{align}
	for all $N\geq 1$.
\end{theorem}

In the theorem above, the constant $C$ depends on $d$, $T$, and the Lipschitz-constants in Assumption \ref{ass:lipschitz.kernel}. 

\begin{remark}
Let us quickly compare this result with its counterpart for the classical Wasserstein distance; we refer to \cite{fournier2015rate} for  general results and  background on the problem:
Ignoring the temporal structure and viewing $\mu$ as the law of a random variable on $[0,1]^{dT}$, one has 
	\[	E\Big[ \mathcal{W}(\mu,\muempirical)\Big]
	\leq C 
	\begin{cases}
	N^{-1/2} &\text{for } dT=1,\\
	N^{-1/2}\log(N+1) &\text{for } dT=2,\\
	N^{-1/(dT)} &\text{for } dT \geq 3,
	\end{cases} \]
for all $N\geq 1$, and these rates are known to be sharp.
As a consequence, for $d\geq 3$ the adapted empirical measure converges in adapted Wasserstein distance at optimal rates. 
For $d=2$ the rates are optimal up to a logarithmic factor and for $d=1$ the rate is (possibly) not optimal, but approaches the optimal one for large $T$.
\end{remark}

Our final main result is the following concentration inequality:

\begin{theorem}[Deviation]
\label{thm:deviation}
	Under Assumption \ref{ass:lipschitz.kernel}, there are constants $c,C>0$ such that 
	\[ P\Big[ \mathcal{AW}(\mu,\mucausal)  \geq C\mathop{\mathrm{rate}}(N)+\varepsilon \Big]
	\leq 2T\exp\Big( -cN\varepsilon^2 \Big) \]
	for all $N\geq 1$ and all $\varepsilon>0$.
\end{theorem}
As above, 
the constants $c,C$ depend on $d$, $T$, and the Lipschitz constants in Assumption \ref{ass:lipschitz.kernel}.

Finally, the following asymptotic regime  consequence of Theorem \ref{thm:deviation} holds true.

\begin{corollary}
\label{cor:deviation.asymptotic}
	There exists a constant $c>0$ such that:
	For every $\varepsilon>0$ there exists $N_0(\varepsilon)$ such that
	\[ P\Big[  \mathcal{AW}(\mu,\mucausal)  \geq\varepsilon \Big]
	\leq 2T \exp\Big( -cN\varepsilon^2 \Big) \]
	for all $N\geq N_0(\varepsilon)$. 
	In particular, 
\[ \limsup_{N\to\infty}  \frac{1}{N}\log P\Big[  \mathcal{AW}(\mu,\mucausal)  \geq\varepsilon \Big]
	\leq -c\varepsilon^2  \]
	for all $\varepsilon>0$.
\end{corollary}

\begin{example}
\label{ex:lipschitz.kernel}
	We provide three simple examples in which Assumption \ref{ass:lipschitz.kernel} on regularity of disintegrations is satisfied. 
	A proof of these statements is given in Section \ref{sec:aux}.	
	\begin{enumerate}[(a)]
	\item
	Assume that $\mu$ is the law of a stochastic process $(X_t)_{t=1,\dots,T}$ which follows the dynamics 
	\[X_{t+1}=F_{t+1}(X_1,\dots,X_t,\varepsilon_{t+1})\]
	for $t=1,\dots,T-1$, with arbitrary $X_1$.
	Here $F_{t+1}\colon ([0,1]^d)^t\times\mathbb{R}^d\to[0,1]^d$ are given functions and $\varepsilon_{t+1}$ is an $\mathbb{R}^d$-valued random variable independent of $X_1,\dots,X_t$.
	If $(x_1,\dots,x_t)\mapsto F_{t+1}(x_1,\dots,x_t,z)$ is $L$-Lipschitz for every $z\in\mathbb{R}^d$, then Assumption \ref{ass:lipschitz.kernel} holds with Lipschitz constant $L$.
	\item
	Assume that the probability $\mu\in\mathrm{Prob}(([0,1]^d)^T)$ has a density $f$ w.r.t.\  Lebesgue measure on $([0,1]^d)^T$.
	If $f$ is $L$-Lipschitz continuous and there is a constant $\delta>0$ for which $f\geq \delta$, then Assumption \ref{ass:lipschitz.kernel} holds with Lipschitz constant $\sqrt{d}2L/\delta$.
	\item
	Complementing the previous point,  Assumption \ref{ass:lipschitz.kernel} holds  if $\mu$ is supported on finitely many points.
	\end{enumerate}
\end{example}

\begin{remark}
\label{rem:Markov}
	Our estimator $\mucausal$ does not require the knowledge of any properties of the underlying probability $\mu$; for instance, the Lipschitz constant only enters the constant before the rate of convergence but not the construction of $\mu$.
	A natural question is whether, if $\mu$ is known to be subject to an additional structure, one can build modifications of $\mucausal$ that are better suited to this structure.
	We provide a result in this direction in Section \ref{sec:Markov}:
	if $\mu$ is known to be Markov, then one can come up with a (obvious) modification of $\mucausal$ that is subject to drastically better statistical properties; see Theorem \ref{thm:rates.Markov}.
\end{remark}

\subsection{Connection with existing literature}
\subsubsection{Adapted topologies}
A number of authors have independently  introduced strengthened variants of the weak topology which take the temporal structure of processes into account.

Aldous \cite{Al81} introduced \emph{extended weak convergence} as a type of convergence of stochastic processes that in particular guarantees continuity of optimal stopping problems. This line of work has been continued in \cite{LaPa90, CoTo07,HoKe84,Ho91,CoMeSl01, Me03}, among others. Applications to stability of SDEs/BSDEs have particularly seen a burst of activity in the last two decades. We refer to the recent article  \cite{PaPoSa18} for an overview of the many available works in this direction.

In the economics literature, Hellwig \cite{He96} introduced the \emph{information topology}. The work of Hellwig \cite{He96} was motivated by questions of stability in dynamic economic models/games; see  \cite{Jo77, VZ02,HeSch02,BarbieGupta} for further research in this direction.
 
  Pflug and Pflug-Pichler \cite{PfPi12, Pi13,  PfPi14, PfPi15, PfPi16} introduced  the \emph{nested distance} and systematically applied it to stochastic multistage optimization problems. Independently, adapted versions of the Wasserstein distance were also considered by R\"uschendorf \cite{Ru85}, Bion-Nadal  and Talay \cite{BiTa19} and Gigli \cite[Chapter 4]{Gi04}. 
 Adapted distances / topologies on laws of processes are of fundamental importance in questions of stability in mathematical finance and stochastic control, see \cite{Do13, BaBeHuKa17, glanzer2019incorporating, AcBaCa18, BaDoDo19, BaBaBeEd19b, BaBePa18}.
  
Notably, all these notions  (and in fact several more that we do not discuss here) define the same topology  in the present discrete time setup, see \cite{BaBaBeEd19a} and the work of Eder \cite{Ed19}.

\subsubsection{Empirical measures and adapted Wasserstein distance}

As mentioned above, it was first noted by Pflug-Pichler in \cite{PfPi16} that for the classical empirical measures $\muempirical$ we may not have $\mathcal{AW}(\mu,\muempirical)\to 0$ a.s. To obtain a viable estimator,  the authors propose to convolute $\muempirical$ with a suitably scaled smoothing  kernel. Provided the density of $\mu$ is sufficiently regular, they obtain  weak consistency in adapted Wasserstein distance \cite[Theorem 25]{PfPi16}.  This is improved upon in \cite[Theorem 4]{glanzer2019incorporating} where also a deviation inequality is obtained. The main assumption in the latter result is the existence of a Lipschitz continuous density for $\mu$, which is bounded away from zero. 
This assumption is in line with  Assumption \ref{ass:lipschitz.kernel} above, needed for the deviation result of the present article in Theorem \ref{thm:deviation}. 
Specifically, \cite[Theorem 4]{glanzer2019incorporating} is a deviation inequality as in Corollary \ref{cor:deviation.asymptotic}, however with $\varepsilon^2$ replaced by $\varepsilon^{2Td+4}$ (which implies slower decay as $\varepsilon^{2Td+4} <\varepsilon^2$ for small $\varepsilon$).
 
 We stress that Theorem \ref{thm:almost.sure.convergence}  does not require further assumptions on the measure $\mu$ and has no predecessor in the literature.

Conceptually, the \emph{convoluted empirical measure} considered in \cite{PfPi16, glanzer2019incorporating}  is related to the adapted empirical measure considered in the present article. A notable difference is that, by construction, the convoluted empirical measure is not discrete and, for practical purposes, a further discretization step may have to be considered in addition to the convolution step. 

\subsection{Organization of the paper}

We start by introducing the required notation in Section \ref{sec:notation}.
The proof of Theorem \ref{thm:rates.unit.cube} is presented in Section \ref{sec:proof.mean} together with some results which will be applied in the later sections.
We then proceed with the proof of Theorem \ref{thm:deviation} in Section \ref{seq:proof.dev}, building on results of the previous section.
The proof of Theorem \ref{thm:almost.sure.convergence} is presented in Section \ref{sec:proof.as.convergence}, and again builds on (all) previous results.
Section \ref{sec:Markov} deals with the modified estimator for Markov processes.
Finally, Section \ref{sec:aux} is devoted to the proof of the examples stated in the introduction.

\section{Notation and preparations} 
\label{sec:notation}

Throughout the paper, we fix $d \in \mathbb{N}$, $T\ge 2$, and let $\mu$ be a probability measure on $([0,1]^d)^T$.
We consider $([0,1]^d)^T$ as a filtered space endowed with the canonical filtration $(\mathcal{F}_t)_t$ which is generated by the coordinate mappings.
For $1\leq t\leq T$ and a Borel set $G\subset([0,1]^d)^t$ we write $\mu(G):=\mu(G\times([0,1]^d)^{T-t})$ (think of $G$ as $\mathcal{F}_t$-measurable). 
Note that 
\begin{align*}
\int f(x)\,\mu(dx)
&=\iint \cdots \int f(x_1,\dots,x_T)\, \mu_{x_1,\dots,x_{T-1}}(dx_T)\cdots \mu_{x_1}(dx_2)\,\mu_1(dx_1)
\end{align*}
for every (bounded measurable) function $f\colon ([0,1]^d)^T\to\mathbb{R}$ which amounts to the tower property for conditional expectations and the definition of $\mu_{x_1,\dots,x_t}$ as the kernels / conditional probabilities.
Here $\mu_1$ is the first marginal of $\mu$, and to ease notation, we make the convention $\mu_{x_1,\dots,x_t}:=\mu_1$ for $t=0$.

We now turn to notation more specific to this paper:
For $\nu\in\mathrm{Prob}(([0,1]^d)^T)$, $1\leq t\leq T-1$, and a Borel set $G\subset ([0,1]^d)^t$, define the averaged (over $G$) kernel
\begin{align}
\label{def:mu.G}
\nu_G(\cdot)
:=\frac{1}{\nu(G)} \int_{G} \nu_{x_1,\dots,x_t}(\cdot)\,\nu(dx)
\in \mathrm{Prob}([0,1]^d)
\end{align}
with an arbitrary convention if we have to divide by 0; say $\nu_G=\delta_0$ in this case.
In other words, if $X\sim\nu$, then 
\[\nu_G(\cdot) =P[X_{t+1}\in \,\cdot\, | (X_1,\dots, X_t)\in G ] \]
is the conditional distribution of $X_{t+1}$ given that $(X_1,\dots, X_t)\in G$.

Next recall the definitions of $\varphi^N$ and $r$ given in the introduction and define 
\[\Phi^N:=\big\{(\varphi^N)^{-1}(\{x\}) : x\in \varphi^N([0,1]^d) \big\},\]
which forms a partition of $[0,1]^d$ associated to $\varphi^N$ such that 
\[ [0,1]^d=\bigcup_{F\in \Phi^N} F \quad\text{disjoint,}
\qquad
\sup_{F\in\Phi^N}\mathop{\mathrm{diam}} (F) \leq \frac{C}{N^r},
\qquad
| \Phi^N| \leq N^{rd}. \]
Here $\mathop{\mathrm{diam}} (F):=\sup_{x,y\in F} |x-y|$ and $| \Phi^N|$ denotes the number of elements in $\Phi^N$.

Then, for every $1\leq t\leq T-1$ and every
\[ G\in \Phi^N_t:=\Big\{ \prod_{1\leq s\leq t} F_s :  F_s\in \Phi^N  \text{ for all } 1\leq s\leq t\Big\}\]
one has 
\begin{align}
\label{eq:expression.kernel.estimator.partition}
	\mucausal[G] 
	=\frac{1}{ \big| \big\{ \substack{n\in\{1,\dots,N\} \text{ s.t.}\\  (X_1^n,\dots,X_t^n)\in G}  \big\} \big| } \sum_{\substack{n\in\{1,\dots,N\} \text{ s.t.}\\  (X_1^n,\dots,X_t^n)\in G}} \delta_{\varphi^N(X^n_{t+1})},
\end{align}
where, as before, we set $\mucausal[G]=\delta_0$ if we have to divide by zero.
Moreover, as $\mucausal$ charges every $G\in\Phi_t^N$ exactly once (at $\varphi^N(G):=\{ \varphi^N(g) : g\in G\}$ which consist of a single point), setting $\mucausal_g:=\mucausal_G$ for $g\in G\in \Phi_t^N$ defines a disintegration of $\mucausal$. 
Finally, let us already point out at this stage that the denominator in front of the sum in \eqref{eq:expression.kernel.estimator.partition} equals $N\muempirical(G)$.

\begin{remark}
\label{rem:classical.empirical.doesnt.converge}
	At least when $\mu$ has a density w.r.t.\ the Lebesgue measure, the probability that two observations coincide at some time is equal to zero, that is, $P[X_t^n=X_t^m \text{ for some } n\neq m \text{ and } 1\leq t\leq T]=0$.
	Therefore the kernels of $\muempirical$ are almost surely Dirac measures,  meaning that if $Y$ is distributed according to $\muempirical$, then the entire (future) evolution of $Y$ is known already at time $1$.
	This implies that the classical empirical measure cannot capture any temporal structure and convergence in the adapted weak topology will not hold true.
	In accordance, the values of multistage stochastic optimization problems (for instance optimal stopping, utility maximization, ...) computed under $\muempirical$ will not converge to the respective value under $\mu$ in general.
	In Example \ref{ex:opt.stop.not.cont.usual.empirical} we illustrate this for the optimal stopping problem.
\end{remark}

	In contrast, we have just seen in \eqref{eq:expression.kernel.estimator.partition}
 that the kernels of our modified empirical measure $\mucausal$ are in general not Dirac measures and in fact behave like {\emph{averaged kernels}} of the empirical measure:
	for a Borel set $G\subset ([0,1]^d)^t$ one has  
	\begin{align}
	\label{eq:expression.kernel.empirical.partition}
	\muempirical[G]= \frac{1}{ \big|\big\{ \substack{n\in\{1,\dots,N\} \text{ s.t.}\\  	(X_1^n,\dots,X_t^n)\in G}  \big\}\big| } \sum_{\substack{n\in\{1,\dots,N\} \text{ s.t.}\\  (X_1^n,\dots,X_t^n)\in G}} \delta_{ X^n_{t+1}},
	\end{align}
	showing that $\mucausal[G]$ is indeed the push forward of $\muempirical[G]$ under $\varphi^N$.
	
	In fact, we will show in Lemma \ref{lem:ingredients.indep} that (conditionally) $\muempirical[G]$ has the same distribution as $\widehat{\mu_G}^{L_G}$, the empirical measure of $\mu_G$ with a random number $L_G:=N\widehat{\mu}^N(G)$ of observations.

In order to exclude the necessity to distinguish whether the random number $L_G$ above is positive or not, it will turn out useful to make the convention that for any probability, its empirical measure with sample size zero is just the Dirac at zero.
In Section \ref{sec:deviation} it is furthermore convenient to denote $\mathcal{G}_t^N:=\{ \muempirical(G) : G\in\Phi^N_t\}$.

In order to lighten notation in the subsequent proofs, we finally define
\begin{align}
\label{eq:def.R}
 R\colon[0,+\infty)\to[0,+\infty],
\quad R(u):=\begin{cases}
	u^{-1/2}&\text{if } d=1,\\
	u^{-1/2}\log(u+3) &\text{if } d=2,\\
	u^{-1/d}&\text{if } d\geq 3.
	\end{cases}
\end{align}
The reason to go with $\log(u+3)$ in the definition of $R$ instead of $\log(u+1)$ as would have been natural in view of the statement of Theorem \ref{thm:rates.unit.cube} is to guarantee that $u\mapsto uR(u)$ is concave, which simplifies notation.
Also set $0R(0):=\lim_{u\to 0} u R(u)=0$. 

Throughout the proofs, $C>0$ will be a generic constant depending on all sorts of external parameters, possibly increasing from line to line; e.g.\ $2C x^2\leq Cx^2$ for all $x\in\mathbb{R}$ but not $2x^2\leq x^2/C$ or $N\leq C$ for all $N$.

\section{Proof of Theorem \ref{thm:rates.unit.cube}}
\label{sec:proof.mean}

We split the proof into a number of lemmas, which we will reference throughout the paper. 
In particular, we will sometimes (but not always) work under Assumption \ref{ass:lipschitz.kernel}  that the kernels of $\mu$ are Lipschitz, that is, there is a constant $L$ such that 
\[\mathcal{W}(\mu_{x_1,\dots,x_t},\mu_{y_1,\dots,y_t})
\leq L |(x_1,\dots,x_t)-(y_1,\dots,y_t)| \]
for all $(x_1,\dots,x_t)$ and $(y_1,\dots,y_t)$ in $([0,1]^d)^t$, and all $1\leq t\leq T-1$.

\begin{lemma}
\label{lem:aw.estimate.lipschitz.kernel}
	Assume that the kernels of $\mu$ are Lipschitz.
	Then there is a constant $C>0$ such that 
	\[ \mathcal{AW}(\mu,\nu)
	\leq C \mathcal{W}(\mu_1, \nu_1 ) + C\sum_{t=1}^{T-1} \int \mathcal{W}(\mu_{y_1,\dots,y_t}, \nu_{y_1,\dots,y_t} )\,\nu(d y)	  \]
	for every $\nu\in\mathrm{Prob}(([0,1]^d)^T)$.
\end{lemma}
\begin{proof}
	We first present the proof for $T=2$ which is notationally simpler:
	Making use of the dynamic programming principle for the adapted Wasserstein distance \cite[Proposition 5.1 and equation (5.1)]{BaBeLiZa16}, we can write
	\begin{align}
	\label{eq:dyn.prog.AW.2period.lipschitz}
	&\mathcal{AW}(\mu,\nu)
	=\inf_{\gamma\in \mathrm{Cpl}(\mu_1, \nu_1)} \int \left[ |x_1-y_1| + \mathcal{W}(\mu_{x_1},\nu_{y_1}) \right ]\gamma(dx_1, dy_1).
	\end{align}
	Calling $L$ the Lipschitz constant of the kernel $x_1\mapsto \mu_{x_1}$, the triangle inequality implies $\mathcal{W}(\mu_{x_1},\nu_{y_1})\leq L |x_1-y_1| + \mathcal{W}(\mu_{y_1},\nu_{y_1})$ for all $x_1,y_1\in [0,1]^d$.
	Plugging this into \eqref{eq:dyn.prog.AW.2period.lipschitz} yields the claim for $T=2$.
	
	In case of $T\geq 2$, recall that $\mu_{x_1,\dots,x_t}:=\mu_1$ and similarly $\nu_{y_1,\dots,y_t}:=\nu_1$ for $t=0$.
	Further write $x_{1:t}:=(x_1,\dots,x_t)$  and $x_{t:T}:=(x_t,\cdots,x_T)$ for $x_1,\dots,x_T$ in $[0,1]^d$ and $1\leq t\leq T$.
	The dynamic programming principle for the adapted Wasserstein distance (see again \cite[Proposition 5.1]{BaBeLiZa16}) asserts that $\mathcal{AW}(\mu,\nu)=V_0$, where $V_T:=0$ and, recursively
	\begin{align*}
	V_t(x_{1:t},y_{1:t})
	:=\inf_{\gamma \in \mathrm{Cpl}(\mu_{x_{1:t}}, \nu_{y_{1:t}}) } \int\Big( |x_{t+1}-y_{t+1}| &+V_{t+1}(x_{1:t+1},y_{1:t+1})\Big) \\
	& \gamma(dx_{t+1}, dy_{t+1})
	\end{align*}
	for $x_{1:t}$ and $y_{1:t}$ in $([0,1]^d)^t$.
	We will prove the claim via backward induction, showing that for all $x_{1:t}$ and $y_{1:t}$ in $([0,1]^d)^t$ it holds that
	\begin{align}
	\label{eq:dyn.prog.induction}
	\begin{split}
	&V_t(x_{1:t},y_{1:t})
	\leq C\Big( |x_{1:t}-y_{1:t}|+\sum_{s=t}^{T-1}\int \mathcal{W}(\mu_{y_{1:s}}, \nu_{y_{1:s}} )\,\bar{\nu}_{y_{1:t}}(dy_{t+1:T}) \Big).
	\end{split}
	\end{align}
	Here $\bar{\nu}_{y_{1:t}}$ is the conditional probability $P[(Y_{t+1},\dots Y_T)\in \cdot | Y_1=y_1,\dots,Y_t=y_t]$ where $Y\sim\nu$
	with the convention that $\bar{\nu}_{y_{1:t}}=\nu$ for $t=0$, that is, $\bar{\nu}_{y_{1:T-1}}:=\nu_{y_{1:T-1}}$ and recursively
	\[\bar{\nu}_{y_{1:t}}(dy_{t+1:T})
	:=\bar{\nu}_{y_{1:t+1}}(dy_{t+2:T}) \nu_{y_{1:t}}(dy_{t+1})\]
	for $t=T-1,\dots,0$ and $y_{1:t}$ in $([0,1]^d)^t$.
	
	For $t=T$, \eqref{eq:dyn.prog.induction} trivially holds true.
	Assuming that \eqref{eq:dyn.prog.induction}  holds true for $t+1$, we compute
	\begin{align*}
	V_t(x_{1:t},y_{1:t})
	&\leq C \inf_{\gamma \in \mathrm{Cpl}(\mu_{x_{1:t}}, \nu_{y_{1:t}} )} 
	\int\Big( \sum_{s=t+1}^{T-1}\int \mathcal{W}(\mu_{ y_{1:s} }, \nu_{ y_{1:s} } )\,\bar{\nu}_{ y_{1:t+1} }(dy_{t+2:T})\\
	& + |x_{1:t+1}-y_{1:t+1}| + |x_{t+1}-y_{t+1}| \Big)\,\gamma(dx_{t+1}, dy_{t+1}).
 	\end{align*}
	By definition we have 
	\[|x_{1:t+1}- y_{1:t+1}| 
	=|x_{1:t} - y_{1:t}| + |x_{t+1}-y_{t+1}|.\]
	Now note that the sum over the Wasserstein distance inside the $\gamma$-integral only depends on $y$.
	Therefore it is independent of the choice of coupling $\gamma$ and we arrive at 
	\begin{align*}
	&V_t(x_{1:t},y_{1:t})
	\leq C \Big( |x_{1:t} - y_{1:t}| + \mathcal{W}(\mu_{x_{1:t}}, \nu_{y_{1:t}}) \\
	&\quad+ \int \sum_{s=t+1}^{T-1}\int \mathcal{W}(\mu_{y_{1:s}}, \nu_{y_{1:s}} )\,\bar{\nu}_{y_{1:t+1}}(dy_{t+2:T})\,\nu_{y_{1:t}}(dy_{t+1}) \Big).
	\end{align*}
	Moreover, by assumption,
	\[\mathcal{W}(\mu_{x_{1:t}}, \nu_{y_{1:t}}) 
	\leq L|x_{1:t} - y_{1:t}| + \mathcal{W}(\mu_{y_{1:t}}, \nu_{y_{1:t}}) \]
	for all $x_{1:t}$ and $y_{1:t}$ in $([0,1]^d)^t$.
	Finally, recalling the definition of $\bar{\nu}$, one has that
	\begin{align*}
	&\iint \mathcal{W}(\mu_{y_{1:s}}, \nu_{y_{1:s}} )\,\bar{\nu}_{y_{1:t+1}}(dy_{t+2:T})\,\nu_{y_{1:t}}(dy_{t+1})
	=\int \mathcal{W}(\mu_{y_{1:s}}, \nu_{y_{1:s}} )\,\bar{\nu}_{y_{1:t}}(dy_{t+1:T})
	\end{align*}
	for every $t+1\leq s\leq T-1$ and $y_{1:t}$ in $([0,1]^d)^t$.
	This concludes the proof of \eqref{eq:dyn.prog.induction}.
	
	The result now follows by setting $t=0$ in \eqref{eq:dyn.prog.induction}.
\end{proof}

\begin{lemma}
\label{lem:integral.kernels.leq.averaged.kernel}	
	The following hold almost surely:
	\begin{enumerate}[(i)]
	\item
	We have 
	\[\mathcal{W}(\mu_1, \mucausal[1] )
	\leq \frac{C}{N^r} + \mathcal{W}(\mu_{1}, \muempirical[1] )\]
	for all $N\geq 1$.
	\item
	If the kernels of $\mu$ are Lipschitz, then we have 
	\begin{align*}
	\int \mathcal{W}(\mu_{y_1,\dots,y_t}, \mucausal[y_1,\dots,y_t] )\,\mucausal(dy)
	\leq \frac{C}{N^r} + \sum_{G\in \Phi^N_t} \muempirical(G) \mathcal{W}(\mu_G,\muempirical[G])
	\end{align*}
	for every $N\geq 1$ and every $1\leq t\leq T-1$.
	\end{enumerate}
\end{lemma}
\begin{proof}
\hfill
	\begin{enumerate}[(i)]
	\item
	The triangle inequality implies
	\[ \mathcal{W}(\mu_1, \mucausal[1] )
	\leq \mathcal{W}(\mu_1, \muempirical[1] )  +  \mathcal{W}( \muempirical[1], \mucausal[1] ).\]
	As $\mucausal[1]$ is the push forward of $\muempirical[1]$ under the mapping $\varphi^N$, we obtain 
	\begin{align}
	\label{eq:push.forward.wasserstein}
	\begin{split}
	\mathcal{W}(\widehat{\mu}_1^N,\mucausal[1])
	&\leq \int |x_{1}-\varphi^N(x_1)|\,\muempirical[1](dx_1)\\
	&\leq \sup_{u\in[0,1]^d}|u-\varphi^N(u)|
	\leq \frac{C}{N^r},
	\end{split}
	\end{align}
	where the last inequality holds by assumption on $\varphi^N$.
	This proves the first claim.

	\item
	For the second claim, fix $1\leq t\leq T-1$.
	In a first step, write 
	\begin{align}
	\label{eq:W.kernel.leq.sum.G}
	\int \mathcal{W}(\mu_{y_1,\dots,y_t}, \mucausal[y_1,\dots,y_t] )\,\mucausal(dy)
	\leq \sum_{G\in \Phi^N_t} \muempirical(G) \sup_{g\in G } \mathcal{W}(\mu_g, \mucausal[g] ),
	\end{align}
	where we used that $\mucausal(G)=\muempirical(G)$ for every $G\in\Phi_t^N$ (note that this relation only holds for $G\in\Phi_t^N$ and, of course, not for general $G\subset([0,1]^d)^t$).
	Recalling that $\mucausal[g]=\mucausal[G]$ for every $g\in G\in\Phi_t^N$, we proceed to estimate 
	\begin{align*}
	\mathcal{W}(\mu_g, \mucausal[g] )
	&\leq \mathcal{W}(\mu_g, \mu_G) + \mathcal{W}(\mu_G,\muempirical[G] )
	+ \mathcal{W}(\muempirical[G], \mucausal[G] ).
	\end{align*}
 	As  $\mucausal[G]$ is the push forward of $\muempirical[G]$ under the mapping $\varphi^N$ (see \eqref{eq:expression.kernel.empirical.partition} and the sentence afterwards), the same argument as in \eqref{eq:push.forward.wasserstein} implies that
	$\mathcal{W}(\muempirical[G], \mucausal[G])\leq C N^{-r}$.
	
	Further, convexity of $\mathcal{W}(\mu_g, \cdot)$ together with the definition of $\mu_G$ in \eqref{def:mu.G} imply that for every $G\in\Phi^N_t$ with $\mu(G)>0$, one has
	\begin{align} 
	\label{eq:estimate.mug.muG}
	\begin{split}
	\mathcal{W}(\mu_g, \mu_G)
	&\leq \frac{1}{\mu(G)} \int_{G}  \mathcal{W}(\mu_g, \mu_{z_1,\dots,z_t}) \,\mu(dz)\\
	&\leq  \sup_{h\in G} \mathcal{W}(\mu_g, \mu_h)
	\leq L \mathop{\mathrm{diam}}(G)
	\leq \frac{LC}{N^r}
	\end{split}
	\end{align}
	for all $g\in G$, where $L$ denotes the Lipschitz constant of $(x_1,\dots,x_t)\mapsto \mu_{x_1, \dots, x_t}$. 
	
	Finally note that for every $G\in\Phi^N_t$ with $\mu(G)=0$ one has $\muempirical(G)=0$ almost surely.
	Hence one may restrict to those $G$  for which $\mu(G)>0$ in the sum on the right hand side of \eqref{eq:W.kernel.leq.sum.G}.
	This completes the proof.
	\qedhere
	\end{enumerate}
\end{proof}

By definition, $\muempirical[1]$ is the empirical measure of $\mu_1$ with $N$ observations. The following lemma shows a similar phenomenon under conditioning rather than under projection.

\begin{lemma}
\label{lem:ingredients.indep}
	Let $1 \le t \leq T-1$. 
	Conditionally on $\mathcal{G}_t^N:=(\muempirical(G) )_{G\in\Phi_t^N} $, the following hold.
	\begin{enumerate}[(i)]
	\item 
	The family $\{ \muempirical[G] : G\in\Phi^N_t\}$ is independent.
	\item
	For every $G\in\Phi_t^N$, the law of  $\muempirical[G] $ is the same as that of $\widehat{\mu_G}^{N\muempirical(G)}$ (the empirical measure of $\mu_G$ with sample size $N\muempirical(G)$).
	\end{enumerate}
\end{lemma}
\begin{proof}
	To simplify notation, we agree that `f.a.\ $G$'  will always mean `for all $G\in\Phi^N_t$' throughout this proof and similar for $H$;  i.e.\ $G$ and $H$ always run through $\Phi_t^N$.
	Let $\Phi_t^N\ni G \mapsto L_G\in\{0,\dots,N\}$ such that $\sum_G L_G=N$.
	Both statements of the lemma then follow if we can show that for every family $(A_G)_{G}$ of measurable subsets of $\mathrm{Prob}([0,1]^d)$ we have
	\begin{align}
	\label{eq:mu.G.indep}
	P\Big[ \muempirical[G]\in A_G \text{ f.a.\ } G \Big| N\muempirical(H)=L_H \text{ f.a.\ } H  \Big]
	= \prod_G P\Big[ \widehat{\mu_G}^{L_G} \in A_G\Big]
	\end{align}
	whenever $P[N\muempirical(H)=L_H \text{ f.a.\ } H]\neq 0$.

	We prove this by rewriting the left-hand side of \eqref{eq:mu.G.indep}. 
	For this we first note that the relation 
	\begin{align}\label{eq:blabla}
	N\muempirical(H)=L_H \text{ f.a.\ } H 
	\end{align}
	 is satisfied if and only if exactly $L_H$ of the random variables $(X^1_{1:t},\dots, X^N_{1:t})$ are contained in $H$ for each $H\in \Phi^N_t$. 
	 For example it could be that 
	 \[ X^1_{1:t}, \dots, X^{N_1}_{1:t}\in H_1,
	 \quad X^{N_1+1}_{1:t}, \dots,X^{N_2}_{1:t}\in H_2, \text{ and so on}\]
	 for some $N_1, N_2, \dots \in \mathbb{N}$ and all $H_1, H_2,\dots \in \Phi^N_t$. 
	 Then the sets $\{1, \dots, N_1\}, \{N_1+1, \dots, N_2\}$ yield a partition of $\{1, \dots, N \}$ corresponding to specific subcollections of the random variables $\{X_1, \dots, X_N\}$. 
	 It will turn out to be easier to think about these partitions instead of properties of $\muempirical$ directly. 
	 Of course there will be more than one partition of $\{1, \dots, N\}$ corresponding to \eqref{eq:blabla} in general, but each of these partitions will be distinct, so that we can easily consider them one at a time.
	More formally, we fix such a family $(A_G)_{G}$ and denote by $\mathcal{I}$ the set of all partitions $(I_G)_G$ of $\{1,\dots,N\}$ such that $|I_G|=L_G$ for every $G$.
	We use the shorthand notation $X^I_{1:t}=\{(X^n_1,\dots,X^n_t) : n\in I\}$ for subsets $I\subset \{1,\dots,N\}$.
	Similarly $X_{t+1}^I:=\{ X^n_{t+1} : n\in I\}$.
	
	\begin{enumerate}[(a)]
	\item
	Fix a partition $(I_G)_G$.
	We first claim that 
	\begin{align}
	\label{eq:mu.G.indep.partition}
	P\Big[ \muempirical[G]\in A_G \text{ f.a.\ } G  \Big| X_{1:t}^{I_H}\subset H \text{ f.a.\ } H \Big]
	=\prod_G P\Big[ \widehat{\mu_G}^{L_G}\in A_G \Big]. 
	\end{align}
	To see this, we realize that on the set $\{ X^{I_H}_{1:t}\subset H \text{ f.a.\ } H\}$ we know exactly which of the $(X_{t+1}^1, \dots, X_{t+1}^N)$ play a role for the definition of $\muempirical[G]$. Indeed, we 
	 note that for all $G$ one has
	\begin{align}
	\label{eq:rep.mu.on.calG}
	\muempirical[G]= \frac{1}{|I_G|}\sum_{n\in I_G} \delta_{X^n_{t+1}}
	\quad\text{on }\{ X^{I_H}_{1:t}\subset H \text{ f.a.\ } H\} .
	\end{align}
	The advantage of representation \eqref{eq:rep.mu.on.calG} is that the dependence of $\muempirical[G]$ on $X^n_1,\dots,X^n_t$ is gone and $\muempirical[G]$ depends solely on $X^{I_G}_{t+1}$.
	As the $(X^n)_n$ are i.i.d.\ and the $I_G$ are disjoint, the pairs of random variables 
	\[\Big( X^{I_G}_{1:t},\frac{1}{|I_G|}\sum_{n\in I_G} \delta_{X^n_{t+1}} \Big)_G
		\quad\text{are independent under } P.\]
	The definition of conditional expectations therefore implies that 
	\begin{align}
	\label{eq:indep.products}
	\begin{split}
	&P\Big[ \muempirical[G]\in A_G \text{ f.a.\ } G  \Big| X^{I_H}_{1:t}\subset H \text{ f.a.\ } H \Big]\\
	&=\frac{\prod_G P\big[ \frac{1}{|I_G|}\sum_{n\in I_G} \delta_{X^n_{t+1}} \in A_G \text{ and } X^{I_G}_{1:t}\subset G \big]}{\prod_G P[X^{I_G}_{1:t}\subset G ]}\\
	&=\prod_G P\Big[ \frac{1}{|I_G|}\sum_{n\in I_G} \delta_{X^n_{t+1}} \in A_G \Big| X^{I_G}_{1:t}\subset G \Big].
	\end{split}
	\end{align}
	Further, for every fixed $G$, given $\{X^{I_G}_{1:t}\subset G\}$, the family $X^{I_G}_{t+1}$ is independent with each $X^n_{t+1}$ being distributed according to $\mu_G$.
	Therefore, given $\{X^{I_G}_{1:t}\subset G\}$, the distribution of $\frac{1}{|I_G|}\sum_{n\in I_G} \delta_{X^n_{t+1}} $ equals the distribution of the empirical measure of $\mu_G$ with sample size $|I_G|=L_G$.
	We conclude that
	\[P\Big[ \frac{1}{|I_G|}\sum_{n\in I_G} \delta_{X^n_{t+1}} \in A_G \Big| X^{I_G}_{1:t}\subset G \Big]
	= P\Big[ \widehat{\mu_G}^{L_G}\in A_G \Big]. \]
	Plugging this equality into \eqref{eq:indep.products} yields exactly \eqref{eq:mu.G.indep.partition}.
	
	\item
	We proceed to prove \eqref{eq:mu.G.indep}.
	As $\{N\muempirical(H)=L_H \text{ f.a.\ } H \}$ is the disjoint union of $\{X^{I_H}\subset H \text{ f.a.\ } H\}$ over $(I_H)_H\in\mathcal{I}$, we deduce that
	\begin{align*}
	&P\Big[ \muempirical[G]\in A_G \text{ f.a.\ } G \Big| N\muempirical(H)=L_H \text{ f.a.\ } H  \Big] \\
	&=\sum_{(I_H)_H\in\mathcal{I}} P\Big[ \muempirical[G]\in A_G \text{ f.a.\ } G \Big| X^{I_H}_{1:t}\subset H \text{ f.a.\ } H \Big]
	\frac{P[X^{I_H}_{1:t}\subset H \text{ f.a.\ } H ] }{P[ N\muempirical(H)=L_H \text{ f.a.\ } H]  }
	\end{align*}
	By \eqref{eq:mu.G.indep.partition}, the first term inside the sum is equal to the product of $\Pi_G P[\widehat{\mu_G}^{L_G}\in A_G]$ and in particular does not depend on the choice of $(I_H)_H$.
	The sum over the fractions equals 1, which shows \eqref{eq:mu.G.indep} and thus completes the proof.
	\qedhere
	\end{enumerate}
\end{proof}

\begin{lemma}
\label{lem:conditional.rate.expectation}
	The following hold.
	\begin{enumerate}[(i)]
	\item
	We have 
	\[E[\mathcal{W}(\mu_1,\muempirical[1])]
	\leq C R\Big( \frac{N}{ N^{rd(T-1)}} \Big)\]
	for all $N\geq 1$.	
	\item
	For every $1\leq t\leq T-1$ we have
	\[ E\Big[\sum_{G\in \Phi^N_t} \muempirical(G) \mathcal{W}(\mu_G,\muempirical[G])\Big| \mathcal{G}_t^N \Big]
	\leq C R\Big( \frac{N}{ N^{rd(T-1)}} \Big) \]
	almost surely for all $N\geq 1$.
	\end{enumerate}
\end{lemma}
\begin{proof}
\hfill
	\begin{enumerate}[(i)]
	\item
	As $\muempirical[1]$ is the empirical measure of $\mu_1$ with $N$ observations,  \cite[Theorem 1]{fournier2015rate} implies that $E[\mathcal{W}( \mu_1,\muempirical[1] )]\leq C R( N)$ for all $N\geq 1$. 
	The claim follows as $R$ is decreasing.
	\item
	For the second claim, fix $1\leq t\leq T-1$.
	Lemma \ref{lem:ingredients.indep} implies that, conditionally on $\mathcal{G}_t^N$, the distribution of each $\muempirical[G]$ equals the distribution of the empirical measure of $\mu_G$ with sample size $N\muempirical(G)$.
	Therefore, estimating the mean speed of convergence of the classical empirical measure by e.g.\ \cite[Theorem 1]{fournier2015rate}, one has that
	\begin{align}
	\label{eq:rates.conditionally.on.muG}
	E\big[\mathcal{W}(\mu_G,\muempirical[G]) \big| \mathcal{G}_t^N \big]
	\leq C R(N\muempirical(G))
	\end{align}
	almost surely for all $N\geq 1$. 
	
	Summing \eqref{eq:rates.conditionally.on.muG} over $G\in\Phi^N_t$ yields
	\begin{align*}
	&E\Big[\sum_{G\in \Phi^N_t} \muempirical(G) \mathcal{W}(\mu_G,\muempirical[G])\Big|\mathcal{G}_t^N \Big] 
	\leq C \sum_{G\in \Phi^N_t}	\muempirical(G)R\Big(N \muempirical(G)\Big) \\
	&=C  \frac{|\Phi^N_t|}{N} \cdot \sum_{G\in \Phi^N_t}	\frac{1}{|\Phi^N_t|} \Big( N\muempirical(G)\Big) R\Big(N \muempirical(G)\Big).
	\end{align*}
	Now, concavity of $u\mapsto uR(u)$ implies that the latter term is smaller than
	\begin{align*}
	 C \frac{|\Phi^N_t|}{N} \cdot  \Big(\sum_{G\in \Phi^N_t}\frac{N \muempirical(G)}{|\Phi^N_t|} \Big) R\Big(\sum_{G\in \Phi^N_t}\frac{N \muempirical(G)}{ |\Phi^N_t| }\Big) \Big)
	=C  R\Big(\frac{N}{ |\Phi^N_t| }\Big).
	\end{align*}
	Finally, using that $R$ is decreasing, one obtains
	\[R\Big( \frac{N}{|\Phi^N_t|} \Big)
	\leq R\Big( \frac{N}{|\Phi^N_{T-1}|} \Big)
	\leq R\Big( \frac{N}{ N^{rd(T-1)}} \Big).\]
	This completes the proof.
	\qedhere
	\end{enumerate}
\end{proof}

\begin{proof}[Proof of Theorem \ref{thm:rates.unit.cube}]
	By Lemma \ref{lem:aw.estimate.lipschitz.kernel} one has that
	\[ E\Big[\mathcal{AW}(\mu,\mucausal)\Big]
	\leq C  E[\mathcal{W}(\mu_1, \mucausal[1] )] + C\sum_{t=1}^{T-1} E\Big[\int \mathcal{W}(\mu_{y_1,\dots,y_t}, \mucausal[y_1, \dots, y_t] )\,\mucausal(d y)\Big].\] 
	Combining Lemma \ref{lem:integral.kernels.leq.averaged.kernel} and Lemma \ref{lem:conditional.rate.expectation} together with the tower property shows that 
	\[ E\Big[\mathcal{AW}(\mu,\mucausal)\Big]
	\leq C\Big( \frac{1}{N^r} + R\Big(\frac{N}{N^{rd(T-1)}}\Big) \Big)\]
	for every $N\geq 1$. 	
	
	Lastly, by definition of $r$, one has that
	\[ \frac{N}{N^{rd(T-1)}}
	=\begin{cases}
	N^{2/(T+1)} &\text{if } d=1,\\
	N^{1/T} &\text{if } d\geq 2.
	\end{cases}\]
	Recalling the definition of $R$ in \eqref{eq:def.R} (and noting that $\log(N^{1/T} +
3) \leq C \log(N + 1)$ for all $N$ in case $d = 2$) yields the claim.
\end{proof}

\section{Proof of Theorem \ref{thm:deviation}} \label{sec:deviation}
\label{seq:proof.dev}

The proof uses the following basic result for subgausian random variables, where we use the convention $x/0=\infty$ for $x>0$ and $0\cdot \infty =0$.

\begin{lemma}
\label{lem:subgaussian}
	For an integrable zero mean random variable $Y$ and $\sigma\geq 0$ consider the following:
	\begin{enumerate}[(i)]
	\item
	$E[\exp(tY)]\leq \exp(t^2\sigma^2/2)$ for all $t\in\mathbb{R}$.
	\item
	$P[|Y|\geq t]\leq 2 \exp(-t^2/(2\sigma^2))$ for all $t\in\mathbb{R}_+$.	
	\end{enumerate}
	Then (i) implies (ii).
	Moreover, (ii) implies (i) with $\sigma^2$ replaced by $8\sigma^2$ in (i).
\end{lemma}
\begin{proof}
	See, for instance, \cite[Proposition 2.5.2]{Vershynin:2018td}.
\end{proof}

A zero mean random variable $Y$ which satisfies part (i) of Lemma \ref{lem:subgaussian} is called subgaussian with parameter $\sigma^2$.
From the definition it immediately follows that if $Y_1,\dots,Y_n$ are independent $\sigma_k^2$-subgaussian random variables, then $\sum_{k=1}^n Y_k$ is again subgaussian with parameter $\sum_{k=1}^n\sigma_k^2$.
In particular, one obtains Hoeffding's inequality
\[ P\Big[ \Big|\sum_{k=1}^n Y_k\Big|\geq t\Big]
\leq 2\exp\Big( \frac{-ct^2}{\sum_{k=1}^n\sigma_k^2}\Big) \quad\text{for all }t\geq 0.\]
The reason why subgaussian random variables are of interest in the proof of Theorem \ref{thm:deviation} is the following:

\begin{lemma}
\label{lem:ingredients.subgauss}
	The following hold.
	\begin{enumerate}[(a)]
	\item
	The random variable
	\[ \mathcal{W}(\mu_1,\muempirical[1]) - E[\mathcal{W}(\mu_1,\muempirical[1])]\]
	is subgaussian with parameter $C/N$	for all $N\geq 1$.
	\item
	Let $1\leq t\leq T-1$ and $G\in\Phi^N_t$. 
	Then, conditionally on $\mathcal{G}_t^N$, the random variable
	\begin{align*}
	\muempirical(G)\Big( \mathcal{W}(\mu_G,\muempirical[G]) - E\big[\mathcal{W}(\mu_G,\muempirical[G])\big|\mathcal{G}_t^N \big] \Big)
	\end{align*}
	is subgaussian with parameter $C\muempirical(G)/N$ for all $N\geq 1$.
	\end{enumerate}
\end{lemma}
\begin{proof}
\hfill
	\begin{enumerate}[(a)]
	\item
	More generally than in the statement of the lemma, let $\nu\in\mathrm{Prob}([0,1]^d)$ and $L\in\mathbb{N}$ be arbitrary.
	Applying McDiarmid's inequality to the function $([0,1]^d)^L\ni x\mapsto \mathcal{W}(\nu,1/L\sum_{n=1}^L \delta_{x^n})$ shows that the random variable
	\[ \mathcal{W}(\nu,\widehat{\nu}^L) - E[\mathcal{W}(\nu,\widehat{\nu}^L)]\]
	is subgaussian with parameter $C/L$ (where $\widehat{\nu}^L$ denotes the empirical measure of $\nu$ with sample size $L$).
	This in particular implies point (a) of this lemma.
	\item
	Conditionally on $\mathcal{G}_t^N$, the distribution of $\muempirical[G]$ is the same as the distribution of the empirical measure of $\mu_G$ with $N\muempirical(G)$ observations, see Lemma \ref{lem:ingredients.indep}.
	By (the proof of) part (a) of this lemma this implies that, conditionally on $\mathcal{G}_t^N$, the random variable 
	\[\mathcal{W}(\mu_G,\muempirical[G]) - E\big[\mathcal{W}(\mu_G,\muempirical[G])\big|\mathcal{G}_t^N\big] \]
	is subgaussian with parameter $C/(N\muempirical(G))$.
	Multiplying a $\sigma^2$-subgaussian random variable by a constant $a\geq 0$ yields a $\sigma^2a^2$-subgaussian random variable.
	This completes the proof.	
	\qedhere
	\end{enumerate}
\end{proof}

\begin{lemma}
\label{lem:concentration.sum.mu.G}
	There is a constant $c>0$ such that the following hold.
	\begin{enumerate}[(a)]
	\item
	We have
	\[ P\Big[ \big| \mathcal{W}(\mu_1,\muempirical[1]) - E[\mathcal{W}(\mu_1,\muempirical[1]) ]\big|\geq \varepsilon \Big]
\leq 2\exp(-c\varepsilon^2 N)\]
	for all $\varepsilon\geq 0$ and all $N\geq 1$.
	\item
	We have
	\[ P\Big[ \Big|\sum_{G\in\Phi^N_t} \muempirical(G)\Big( \mathcal{W}(\mu_G,\muempirical[G]) - E[\mathcal{W}(\mu_G,\muempirical[G])| \mathcal{G}_t^N ]\Big) \Big|\geq \varepsilon \Big| \mathcal{G}_t^N \Big]
\leq 2\exp(-c\varepsilon^2 N)\]
	almost surely for all $\varepsilon\geq 0$, all $N\geq 1$, and all $1\leq t\leq T-1$.
	\end{enumerate}
\end{lemma}
\begin{proof}
\hfill
	\begin{enumerate}[(a)]
	\item
	The proof follows immediately from Lemma \ref{lem:subgaussian} and Lemma \ref{lem:ingredients.subgauss}.
	\item
	For later reference, define 
	\begin{align}
	\label{eq:def.Delta.t}
	\Delta_{t}^N&:=\sum_{G\in\Phi^N_t} \Delta_{t,G}^N,\quad\text{where}\\
	\nonumber 
	\Delta_{t,G}^N&:=\muempirical(G)\Big( \mathcal{W}(\mu_G,\muempirical[G]) - E[\mathcal{W}(\mu_G,\muempirical[G])|\mathcal{G}_t^N]\Big)
	\end{align}
	for every $G\in\Phi^N_t$.
	Conditionally on $\mathcal{G}_t^N$, Lemma \ref{lem:ingredients.indep} and Lemma \ref{lem:ingredients.subgauss} imply that $\{\Delta_{t,G}^N: G\in\Phi^N_t\}$ is an independent family of $C\muempirical(G)/N$-subgaussian random variables.
	Hence, conditionally on $\mathcal{G}_t^N$, the random variable $\Delta_{t}^N$	is subgaussian with parameter $\sum_{G\in\Phi^N_t} C\muempirical(G)/N = C/N$.
	We again use Lemma \ref{lem:subgaussian} to conclude the proof.
	\qedhere	
\end{enumerate}
\end{proof}

\begin{proof}[Proof of Theorem \ref{thm:deviation}]
	By Lemma \ref{lem:aw.estimate.lipschitz.kernel} and Lemma \ref{lem:integral.kernels.leq.averaged.kernel} one has that
	\begin{align*}
	\mathcal{AW}(\mu,\mucausal)
	&\leq C\Big( \frac{1}{N^r}+ \mathcal{W}(\mu_1,\muempirical[1])  +\sum_{t=1}^{T-1} \sum_{G\in\Phi^N_t} \muempirical(G) \mathcal{W}(\mu_G,\muempirical[G]) \Big).
	\end{align*}
	Recalling the definition of $\Delta_{t}^N$ given in  \eqref{eq:def.Delta.t} and setting
	\begin{align}
	\label{eq:def.Delta.0}
	\Delta_0^N&:=\mathcal{W}(\mu_1,\muempirical[1])- E[\mathcal{W}(\mu_1,\muempirical[1])], \\
	\nonumber
	S^N&:=E[\mathcal{W}(\mu_1,\muempirical[1])]+ \sum_{t=1}^{T-1}\sum_{G\in\Phi^N_t} \muempirical(G) E[\mathcal{W}(\mu_G,\muempirical[G])|\mathcal{G}_t^N]
	\end{align}
	for every $N\geq 1$, we can write
	\begin{align*}
	\mathcal{AW}(\mu,\mucausal)
	&\leq C\Big( \frac{1}{N^r} + \sum_{t=0}^{T-1} \Delta_t^N  + S^N \Big).
	\end{align*}
	By Lemma \ref{lem:conditional.rate.expectation} one has that $S^N \leq C R( N/N^{rd(T-1)} )$ almost surely for every $N\geq 1$.
	Recalling that $R(N/N^{rd(T-1})\leq C \mathop{\mathrm{rate}}(N)$ 	and that $N^{-r}\leq\mathop{\mathrm{rate}}(N)$ (with equality for dimension $d\neq 2$), we arrive at
	\begin{align}
	\label{eq:AW.smaller.three.terms}
	\begin{split}
	\mathcal{AW}(\mu,\mucausal)
	&\leq  C\Big( \sum_{t=0}^{T-1} \Delta_t^N  + \mathop{\mathrm{rate}}(N)\Big).
	\end{split}
	\end{align}
	
	Finally, Lemma \ref{lem:concentration.sum.mu.G} (and the tower property) imply that $P[|\Delta_t^N|\geq \varepsilon]\leq 2\exp(-cN\varepsilon^2) $ for all $0\leq t\leq T-1$, all $\varepsilon>0$, and all $N\geq 1$.
	Therefore a union bound shows that
	\begin{align*}
	P\Big[ \mathcal{AW}(\mu,\mucausal)  \geq C\mathop{\mathrm{rate}}(N) + \varepsilon \Big]
	&\leq P\Big[\sum_{t=0}^{T-1} \Delta_t^N  \geq\frac{\varepsilon}{C} \Big] \\
	&\leq 2T\exp( -c N \varepsilon^2)
	\end{align*}
	for all $N\geq 1$ and all $\varepsilon>0$, where $c>0$ is some new (small) constant.
	This completes the proof.
\end{proof}

\begin{proof}[Proof of Corollary \ref{cor:deviation.asymptotic}]
For fixed $\varepsilon>0$, choose $N_0(\varepsilon)$ so that $C\mathop{\mathrm{rate}}(N)\leq \varepsilon$ for all $N\geq N_0(\varepsilon)$ and apply Theorem \ref{thm:deviation}.
\end{proof}

\section{Proof of Theorem \ref{thm:almost.sure.convergence}} 
\label{sec:proof.as.convergence}

We start by proving almost sure convergence of $\mathcal{AW}(\mu,\mucausal)$ to zero under the additional assumption that the kernels of $\mu$ admit a continuous version\footnote{That is, a version such that for all $1\leq t\leq T-1$ the map $([0,1]^{d})^t\ni(x_1,\dots,x_t)\mapsto \mu_{x_1,\dots,x_t}\in \mathrm{Prob}([0,1]^d)$  is continuous, the domain equipped with the associated Euclidean topology and the image with the weak topology.}.
Under this assumption, we can make use of the previous results and conclude almost sure convergence from the deviation inequality and a Borel-Cantelli argument.
At the end of this section we show how this restriction can be removed.

\begin{lemma}
\label{lem:aw.estimate.continuous}
	Assume that the kernels of $\mu$ are continuous.
	Then, for every $\delta>0$ there is a constant $C(\delta)>0$ such that 
	\[ \mathcal{AW}(\mu,\nu)
	\leq \delta + C(\delta) \mathcal{W}(\mu_1,\nu_1) + C(\delta) \sum_{t=1}^{T-1} \int \mathcal{W}(\mu_{y_1,\dots,y_t}, \nu_{y_1,\dots,y_t} )\,\nu(d y)	  \]
	for every $\nu\in\mathrm{Prob}(([0,1]^d)^T)$.
\end{lemma}
\begin{proof}
	The proof is similar to the proof of Lemma \ref{lem:aw.estimate.lipschitz.kernel}.
	For (notational) simplicity we spare the induction and restrict to $T=2$; the general case follows just as in Lemma \ref{lem:aw.estimate.lipschitz.kernel}. 
	For $T=2$ the recursive formula of $\mathcal{AW}(\mu, \nu)$ reads 
	\begin{align}
	\label{eq:dyn.prog.AW.2period}
	&\mathcal{AW}(\mu,\nu)
	=\inf_{\gamma\in \mathrm{Cpl}(\mu_{1}, \nu_1)} \int \left [ |x_1-y_1| + \mathcal{W}(\mu_{x_1},\nu_{y_1}) \right ]\gamma(dx_1, dy_1).
	\end{align}
	Now fix $\delta>0$.
	By uniform continuity of $x_1\mapsto \mu_{x_1}$ (as a continuous function with compact domain), there is $C(\delta)>0$ such that 
	\begin{align*}
	\mathcal{W}(\mu_{x_1},\nu_{y_1}) 
	&\leq \mathcal{W}(\mu_{x_1},\mu_{y_1})  + \mathcal{W}(\mu_{y_1},\nu_{y_1})\\
	&\leq \delta + C(\delta)|x_1-y_1| + \mathcal{W}(\mu_{y_1},\nu_{y_1})
	\end{align*}
	for all $x_1,y_1\in[0,1]^d$.
	Plugging this into \eqref{eq:dyn.prog.AW.2period} yields the claim.
\end{proof}

\begin{lemma}
\label{lem:integral.kernels.leq.averaged.kernel.continuous}
	Assume that the kernels of $\mu$ are continuous and let $1\leq t\leq T-1$.
	Then, for every $\delta>0$ there is a number $N_0(\delta)$ such that 
	\begin{align*}
	\int \mathcal{W}(\mu_{y_1,\dots,y_t}, \mucausal[y_1,\dots,y_t])\,\mucausal(dy)
	\leq \delta + C \sum_{G\in \Phi^N_t} \muempirical(G) \mathcal{W}(\mu_G,\muempirical[G])
	\end{align*}
	almost surely for every $N\geq N_0(\delta)$.
\end{lemma}
\begin{proof}
	The proof follows exactly as in the proof of Lemma \ref{lem:integral.kernels.leq.averaged.kernel}; one only needs to replace the estimate `$\mathcal{W}(\mu_g, \mu_G)\leq CN^{-r}$ for all $g\in G$' (this is \eqref{eq:estimate.mug.muG} within that lemma) by the following:
	
	Let $\delta>0$.
	By uniform continuity of $(x_1,\dots,x_t) \mapsto \mu_{x_1,\dots,x_t}$, there exists $\varepsilon>0$ such that for every $x,y\in([0,1]^d)^t$ with $|x-y|\leq\varepsilon$, one has that $\mathcal{W}(\mu_x, \mu_y)\leq\delta$.
	Now note that for arbitrary $G\in\Phi^N_t$ and $g\in G$ it holds that
	\begin{align*} 
	\mathcal{W}(\mu_g, \mu_G)
	&\leq \frac{1}{\mu(G)} \int_{G}  \mathcal{W}(\mu_g, \mu_{z_1,\dots,z_t}) \,\mu(dz)
	\leq  \sup_{z\in G} \mathcal{W}(\mu_g, \mu_z) 
	\leq \delta,
	\end{align*}
	where the last inequality holds once $\mathop{\mathrm{diam}}(G) \leq \varepsilon$. 
	As $\mathop{\mathrm{diam}}(G) \leq C N^{-r}$ uniformly over $G\in\Phi_t^N$, this concludes the proof. 
\end{proof}

\begin{lemma} 
\label{lem:a.s.convergence.continuous}
	Assume that the kernels of $\mu$ are continuous.
	Then $\mathcal{AW}(\mu,\mucausal) \to 0$ almost surely.
\end{lemma}
\begin{proof}
	The first part of the proof follows the proof of Theorem \ref{thm:deviation}:
	let $\delta>0$ be arbitrary.
	Then, substituting Lemma \ref{lem:aw.estimate.continuous} for Lemma \ref{lem:aw.estimate.lipschitz.kernel} and Lemma \ref{lem:integral.kernels.leq.averaged.kernel.continuous} for Lemma \ref{lem:integral.kernels.leq.averaged.kernel} in the proof of Theorem \ref{thm:deviation}, we conclude that there exist $C(\delta)$ and $N_0(\delta)$ such that
	\begin{align*}
	\mathcal{AW}(\mu,\mucausal)
	&\leq  \delta +C(\delta)\Big( \sum_{t=0}^{T-1} \Delta_t^N + \mathop{\mathrm{rate}}(N)\Big),
	\end{align*}
	almost surely  for all $N\geq N_0(\delta)$; compare with \eqref{eq:AW.smaller.three.terms}.
	Recall that $\Delta_t^N$ was defined in \eqref{eq:def.Delta.t} for $1\leq t\leq T-1$ and in \eqref{eq:def.Delta.0} for $t=0$.
	
	An application of Lemma \ref{lem:concentration.sum.mu.G} then shows that, similar to before, 
	\begin{align*}
	P\Big[ \mathcal{AW}(\mu,\mucausal)  \geq \delta + \varepsilon \Big]
	&\leq P\Big[  C(\delta) \sum_{t=0}^{T-1} \Delta_t^N \geq \varepsilon- C(\delta) \mathop{\mathrm{rate}}(N)\Big] \\
	&\leq 2T \exp\Big( -c N \big( \frac{\varepsilon}{C(\delta)} - \mathop{\mathrm{rate}}(N)\big)^2_+ \Big)
	\end{align*}	
	for all $N\geq N_0(\delta)$, where $c>0$ is some small constant.
	
	Let $N_1(\delta,\varepsilon)$ such that $\mathop{\mathrm{rate}}(N) \leq \varepsilon/(2C(\delta))$ for all $N\geq N_1(\delta,\varepsilon)$.
	Then 
	\begin{align*}
	P\Big[ \mathcal{AW}(\mu,\mucausal) \geq \delta +\varepsilon\Big]
	&\leq 2T \exp\Big( \frac{-c N\varepsilon^2}{4} \Big)
	\end{align*}
	for all $N\geq \max\{N_0,N_1\}$.
	By a Borel-Cantelli argument, this implies that 
	\[ P\Big[ \limsup_{N\to\infty}\mathcal{AW}(\mu,\mucausal) \geq \delta +\varepsilon\Big]=0.\]
	As $\varepsilon,\delta>0$ were arbitrary, we conclude that $\mathcal{AW}(\mu,\mucausal)$ converges to zero almost surely when $N\to\infty$.
	This completes the proof.
\end{proof}

With this preparatory work carried out, we are now ready to prove the strong consistency of $\mucausal$.

\begin{proof}[Proof of Theorem \ref{thm:almost.sure.convergence}]
	We provide the proof for a two-period setting, that is, $T=2$.
	The general case follows by the same arguments, however it involves a (lengthy) backward induction just as in the proof of Lemma \ref{lem:aw.estimate.lipschitz.kernel} and offers no new insights.
	
	Let $\varepsilon>0$.
	We shall construct $\nu\in\mathrm{Prob}(([0,1]^d)^2)$ with continuous conditional probabilities such that $\mathcal{AW}(\mu,\nu)\leq\varepsilon$ and $\limsup_{N}\mathcal{AW}(\mucausal,\nucausal)\leq\varepsilon$ almost surely.
	As $\lim_N\mathcal{AW}(\nu,\nucausal)=0$ almost surely by Lemma \ref{lem:a.s.convergence.continuous}, the triangle inequality then implies that $\limsup_N \mathcal{AW}(\mu,\mucausal)\leq 2\varepsilon$ almost surely.
	Recalling that $\varepsilon>0$ was arbitrary completes the proof.
	\begin{enumerate}[(a)]
	\item
	By Lusin's theorem there is a compact set $K\subset[0,1]^d$ such that $\mu(K)\geq 1-\varepsilon$ and $K\ni x_1\mapsto \mu_{x_1}$ is continuous.
	Extend the latter mapping to a continuous mapping $[0,1]^d\ni x_1\mapsto \nu_{x_1}$ by Tietze's extension theorem (actually, a generalization thereof to vector valued functions: Dugundji's theorem \cite[Theorem 4.1]{dugundji1951extension}) and define
	\[\nu(dx_1,dx_2):=\mu_1(dx_1)\nu_{x_1}(dx_2)\in\mathrm{Prob}(([0,1]^d)^2).\]
	Then, taking the identity coupling $\gamma\in\mathrm{Cpl}(\mu_1,\nu_1)$ (that is, $\gamma=[x_1\mapsto (x_1,x_1)]_\ast \mu_1$) implies that $\mathcal{AW}(\mu,\nu)\leq \int \mathcal{W}(\mu_{x_1},\nu_{x_1})\,\mu_1(dx_1)\leq \varepsilon$.
	\item
	It remains to construct an i.i.d.\ sample of $\nu$ such that  $\limsup_{N}\mathcal{AW}(\mucausal,\nucausal)\leq\varepsilon$.
	To that end recall that $(X^n)_n$ is an i.i.d.\ sample of $\mu$, and define 
	\[Y_1^n:=X_1^n 
	\quad\text{and}\quad
	Y_2^n:=X_2^n 1_{X_1^n \in K} + Z_2^n 1_{X_1^n \notin K}\]
	for every $n$, 	where $Z_2^n$ satisfies that $P[Z_2^n\in \cdot |X_1^n]=\nu_{X_1^n}(\cdot)$ (and $Z^n$ is independent of $\{X^m,Y^m,Z^m : n\neq m\}$).
	Note that $(Y_1^n,Y^n_2)_n$ is an i.i.d.\ sample of $\nu$.
	
	We again take the identity coupling between $\mucausal_1=\nucausal_1$ to obtain
	\begin{align}
	\label{eq:AW.estimate.noncontinuous.kernel}
	 \mathcal{AW}(\mucausal,\nucausal)
	\leq \int \mathcal{W}(\mucausal_{x_1},\nucausal_{x_1})\,\mucausal_1(dx_1)
	=\sum_{G\in\Phi^N_1} \muempirical(G) \mathcal{W}(\mucausal_G,\nucausal_G).
	\end{align}
	In the (second) equality we also used that $\mucausal(G)=\muempirical(G)$ for every $G\in\Phi^N_1$ and that the kernels of $\mucausal$ and $\nucausal$ are constant on every $G\in\Phi_1^N$; in fact
	\begin{align*}
	\mucausal_G&= \frac{1}{N\muempirical(G)} \sum_{n\leq N \text{ s.t.\ } X_1^n\in G} \delta_{\varphi^N(X_2^n)},\\
	\nucausal_G&=\frac{1}{N\muempirical(G)}\sum_{n\leq N \text{ s.t.\ } X_1^n\in G} \delta_{\varphi^N(Y_2^n)}.
	\end{align*}	
	Therefore, making use of convexity of $\alpha,\beta\mapsto \mathcal{W}(\alpha,\beta)$, we further estimate
	\begin{align*}
	\mathcal{W}(\mucausal_G,\nucausal_{G})
	&\leq \frac{1}{ N\muempirical(G)} \sum_{n \leq N \text{ s.t.\ } X_1^n\in G} \mathcal{W}(\delta_{\varphi^N(X_2^n)}, \delta_{\varphi^N(Y_2^n)} ) \\
	&\leq \frac{N\muempirical(G\cap K^c)}{ N\muempirical(G)}\,\text{diam}([0,1]^d),
	\end{align*}
	where we used that $Y_2^n=X_2^n$ whenever $X_1^n\in K$.
	Plugging this estimate into \eqref{eq:AW.estimate.noncontinuous.kernel} yields $\mathcal{AW}(\mucausal, \nucausal) \leq \muempirical(K^c)$.
	To conclude use the strong law of large numbers which guarantees that $\lim_N \muempirical(K^c)=\mu(K^c)\leq \varepsilon\,\text{diam}([0,1]^d)$ almost surely, where the last inequality holds by choice of $K$.
	\qedhere
	\end{enumerate}
\end{proof}

\section{Modified estimator for Markov-processes}
\label{sec:Markov}

In this section we take up Remark \ref{rem:Markov} and we show that under additional structural assumption of $\mu$, it is possible to come up with modified estimators which have improved statistical properties\footnote{The authors would like to thank the anonymous referee for suggesting to investigate this question.}.
We shall do so in the (arguably) most relevant case that $\mu$ is known to be Markov (meaning that $\mu_{x_1,\dots,x_t}$ depends only on $x_t$ and we use the shorthand  notation $\mu_{t,x_t}$ for the latter). 

Set $r:=1/3$ for $d=1$ as well as $r:=1/2d$ for $d\geq 2$ and recall $\varphi^N$ from Definition \ref{def:adapted.empirical.measure}.
Define $\mucausal$ via the formula 
\[ \mucausal := \mucausal_1(dx_1) \, \mucausal_{1,x_1}(dx_2)\,\cdots\,\mucausal_{T-1,x_{T-1}} (dx_T)\]
where, for every $1\leq t\leq T-1$ and $x\in[0,1]^d$ we set
\begin{align*}
\mucausal_1
 &:= \frac{1}{N } 
	\sum_{ n\in\{1,\dots,N\} }  \delta_{\varphi^N(X^n_{1})} \\
\mucausal_{t,x}
&:=
	\frac{1}{\big| \substack{n\in\{1,\dots,N\} \text{ s.t.}\\  \varphi^N(X_t^n)= \varphi^N(x)} \big| } 
	\sum_{ \substack{n\in\{1,\dots,N\} \text{ s.t.}\\  \varphi^N(X_t^n)= \varphi^N(x)}}  \delta_{\varphi^N(X^n_{t+1})} .
\end{align*}
Then, by definition, $\mucausal$ is Markovian too.

\begin{theorem}[Markov]
\label{thm:rates.Markov}
	Assume that $\mu$ is Markov and that it satisfies Assumption \ref{ass:lipschitz.kernel}.
	Then there are two constants $C,c>0$ such that 
	\begin{align*}
	E\Big[ \mathcal{AW}(\mu,\mucausal)\Big]
	&\leq C \mathop{\mathrm{rate}}(N)
	:= C
	\begin{cases}
	N^{-1/3} &\text{for } d=1,\\
	N^{-1/4}\log(N) &\text{for } d=2,\\
	N^{-1/2d} &\text{for } d \geq 3,
	\end{cases}
	\end{align*}
	and 
	\begin{align*}
	P\Big[ \mathcal{AW}(\mu,\mucausal)  \geq C \mathop{\mathrm{rate}}(N)+\varepsilon \Big]
	&\leq 2T\exp\Big( -cN\varepsilon^2 \Big)
	\end{align*}
	for all $N\geq 1$.
\end{theorem}

In the theorem above, the constants $C,c$ depend on $d$, $T$, and the Lipschitz-constants in Assumption \ref{ass:lipschitz.kernel}. 

\begin{remark}
	It is important to stress that for $T>2$, the rate obtained in Theorem \ref{thm:rates.Markov} improves the rates in Theorem \ref{thm:rates.unit.cube} and Theorem \ref{thm:deviation}.
	In particular, the dependence of the rate on $T$ disappears in Theorem \ref{thm:rates.Markov}.
\end{remark}

\begin{proof}[Proof of Theorem \ref{thm:rates.Markov}]
	The proof follows the same lines as the proof for Theorem \ref{thm:rates.unit.cube} and \ref{thm:deviation}.
	The (heuristic) reason for the improved rate is the following: 
	in order to estimate the kernels, there is now no need to partition the whole state space of the past into small cubes but only the last step, which yields a larger number of samples that can be used to estimate the kernels.
	
	We shall only sketch the required modifications, the proofs are essentially the same.
	
	Lemma \ref{lem:aw.estimate.lipschitz.kernel} remains unchanged, noting that Markovianity of $\mu$ and $\mucausal$ implies
	\[ \mathcal{AW}(\mu,\mucausal)\leq C\mathcal{W}(\mu_1,\mucausal_1) + C\sum_{t=1}^{T-1} \int \mathcal{W}(\mu_{t,y_t},\mucausal_{t,y_t}) \,\mucausal(dy). \]
	Further, setting 
	\[ \Phi_t^N:=\Big\{ ([0,1]^d)^{t-1} \times F : F\in\Phi^N \Big\}, \]
	the statements made in Lemma \ref{lem:integral.kernels.leq.averaged.kernel} and Lemma \ref{lem:ingredients.indep} remain the same.
	In Lemma \ref{lem:conditional.rate.expectation}, we obtain $E[\mathcal{W}(\mu_1,\mu_1^N)]\leq CR(N)$ (with the same proof) and
	\[ E\Big[\sum_{G\in \Phi^N_t} \muempirical(G) \mathcal{W}(\mu_G,\muempirical[G])\Big| \mathcal{G}_t^N \Big]
	\leq C R\Big( \frac{N}{ N^{rd}} \Big) .\]
	Indeed, the only change in the proof is that now $|\Phi^N_t|=N^{rd}$ independently of $t$ (instead of $|\Phi^N_t|=N^{trd}$).
	
	Combining Lemma \ref{lem:ingredients.indep} and Lemma \ref{lem:conditional.rate.expectation} then shows that
	\begin{align*}
	E[ \mathcal{AW}(\mu,\mucausal) ]
	\leq C\Big( \frac{1}{N^r} + R\Big( \frac{N}{N^{rd}} \Big) \Big)
	\leq C \mathop{\mathrm{rate}}(N),
	\end{align*}
	where the last equality holds by definition of $R$ and as $N/N^{rd}=N^{\frac{2}{3}}$ for $d=1$ and $N/N^{rd}=N^{\frac{1}{2}}$ for $d\geq 2$.
	This proves the statement pertaining the average rate of convergence.
	The proof for the statement pertaining deviation from the average speed does not require any changes.
\end{proof}

\section{Auxiliary results}
\label{sec:aux}

We start by providing a simple example showing that optimal stopping evaluated at the empirical measure does not converge to the value of the problem under the true model. This was first observed in \cite[Proposition 1]{PfPi16} by Pflug and Pichler.

\begin{example}
\label{ex:opt.stop.not.cont.usual.empirical}
	Consider a Gaussian random walk in two periods, that is, $X_0=0$, $X_1$ and $X_2-X_1$ have standard normal distribution and $X_2-X_1$ is independent of $X_1$. Denote by $\mu$ the law of this random walk, i.e.\ $\mu=\text{Law}(X_0,X_1,X_2)$. 	A classical optimal stopping problem consists of minimizing the expected cost $\int c(\tau,\cdot)\,d\mu$ over all stopping times $\tau\colon\mathbb{R}^3\to\{0,1,2\}$ (here stopping times simply means that $1_{\tau=0}$ is a function of $x_0$ and $1_{\tau=1}$ is a function of $x_0,x_1$ only), where $c\colon\{0,1,2\}\times\mathbb{R}^3\to\mathbb{R}$ is a given cost function.
	
	Now consider the same problem under the empirical measure $\widehat{\mu}^N$ in place of $\mu$ and take for instance the cost function $c(t,x):=x_t$.
	As $X_1$ has Lebesgue density, it follows that $P[X_1^n=X_1^m \text{ for some } n\neq m]=0$ which means that, almost surely, the knowledge of $X_1^n$ gives perfect knowledge of $X_2^n$.
	In particular, for every $N\geq 1$ and almost all $\omega$, the mapping
	\[ \tau^{N,\omega}(x_1):=\begin{cases}
	1 &\text{if } x_1=X_1^n(\omega)\text{ for some } n\leq N \text{ with } X_1^n(\omega)<X_2^n(\omega) \\
	2 &\text{else}
	\end{cases} \]
	defines a stopping time.
	Making use of the strong law of large numbers, we then obtain
	\[ \inf_\tau \int c(\tau,\cdot) \, d\widehat{\mu}^N
	\leq \int c(\tau^N,\cdot) \, d\widehat{\mu}^N
	= \int x_1\wedge x_2\,\widehat{\mu}^N(dx)
	\to\int x_1\wedge x_2\,\mu(dx)
	<0\]
	almost surely.
	This shows that any reasonable type of convergence (almost sure, in probability,...) towards $\inf_\tau \int c(\tau,\cdot)\,d\mu=0$ fails.
	
	The Gaussian framework was chosen for notational convenience, the same result of course applies to absolutely continuous probabilities on the unit cube as well.
\end{example}

We now provide the following proof.

\begin{proof}[Proof of Example \ref{ex:lipschitz.kernel}]
	\hfill
	\begin{enumerate}[(a)]
	\item
	Fix $1\leq t\leq T-1$ and let $(x_1,\dots,x_t)$, $(\tilde{x}_1,\dots,\tilde{x}_t)$ be two elements of $([0,1]^d)^t$.
	Define $\gamma\in \mathrm{Cpl}(\mu_{x_1,\dots,x_t},\mu_{\tilde{x}_1,\dots,\tilde{x}_t})$ by
	\[\gamma(A) := P\big[ \big( F_{t+1}(x_1,\dots,x_t,\varepsilon_{t+1}), F_{t+1}(\tilde{x}_1,\dots,\tilde{x}_t,\varepsilon_{t+1}) \big) \in A \big] \]
	for Borel $A\subset[0,1]^d\times[0,1]^d$.
	Then the assumption made on $F_{t+1}$ yields 
	\begin{align*}
	\mathcal{W}(\mu_{x_1,\dots,x_t},\mu_{\tilde{x_1},\dots,\tilde{x_t}})
	&\leq \int |a-b|\,\gamma(da,db)\\
 	&\leq L|(x_1,\dots,x_t) - (\tilde{x}_1,\dots,\tilde{x}_t)|,
	\end{align*}
	showing that Assumption \ref{ass:lipschitz.kernel} is indeed satisfied.
	\item
	Again fix $1\leq t\leq T-1$ and let $(x_1,\dots,x_t)$, $(\tilde{x}_1,\dots,\tilde{x}_t)$ be two elements of $([0,1]^d)^t$.
	Then $\mu_{x_1,\dots,x_t}$ has the density 
	\[ f_{X_{t+1}|X_1,\dots,X_t}(\cdot):=\frac{ f_{X_1,\dots,X_{t+1}}(x_1,\dots,x_t,\cdot)}{f_{X_1,\dots,X_t}(x_1,\dots,x_t)} \]
	w.r.t.\ the Lebesgue measure on $[0,1]^d$, where $f_{X_1,\dots,X_t}$ denotes the density of the distribution of $(X_1,\dots,X_t)$; similarly for $f_{X_1,\dots,X_{t+1}}$.
	The same goes for $\mu_{\tilde{x}_1,\dots,\tilde{x}_t}$ if $x_s$ is replaced by $\tilde{x}_s$ everywhere.
	Moreover, it is not hard to show that $\mathcal{W}(gdx,hdx)\leq \sqrt{d} \int |g(x)-h(x)|\,dx$ whenever $g$ and $h$ are two Lebesgue-densities on $[0,1]^d$; use e.g.\ the Kantorovich-Rubinstein duality and H\"older's inequality or apply \cite[Theorem 6.13]{villani2008optimal}.
	Therefore one has that
	\begin{align*}
	&\mathcal{W}(\mu_{x_1,\dots,x_t},\mu_{\tilde{x_1},\dots,\tilde{x}_t}) \\
	&\leq \sqrt{d} \int \Big|\frac{ f_{X_1,\dots,X_{t+1}}(x_1,\dots,x_t,u)}{f_{X_1,\dots,X_t}(x_1,\dots,x_t)} -\frac{ f_{X_1,\dots,X_{t+1}}(\tilde{x}_1,\dots,\tilde{x}_t,u)}{f_{X_1,\dots,X_t}(\tilde{x}_1,\dots,\tilde{x}_t)} \Big| \,du.
	\end{align*}
	A quick computation using the assumptions imposed on $f$ shows that the latter can be bounded by $\sqrt{d}2L/\delta$, which completes the proof.
	\item
	In the case that $\mu$ is supported on finitely many points, the disintegration is uniquely defined by its value on these points.
	In particular, any Lipschitz continuous extension of this mapping will do, see e.g.\ \cite{johnson1986extensions}.
	\qedhere
	\end{enumerate}
\end{proof}

\vspace{1em}
\noindent
\textsc{Acknowledgments:} 
%
Daniel Bartl is grateful for financial support through the Vienna Science and Technology Fund (WWTF) project MA16-021 and the Austrian Science Fund (FWF) project P28661.\\
Mathias Beiglb\"ock is grateful for financial support through the Austrian Science Fund (FWF) under project Y782.\\
Johannes Wiesel acknowledges support by the German National Academic Foundation.

\bibliographystyle{abbrv}
\bibliography{joint_biblio}

\end{document}